\documentclass[12pt,draft,reqno]{amsart}
\usepackage{amsmath}
\usepackage{amssymb}
\usepackage{amsthm}
\usepackage{bm}
\usepackage{url}
\numberwithin{equation}{section}
\newtheorem{theorem}{Theorem}[section]
\newtheorem{lemma}[theorem]{Lemma}
\newtheorem{proposition}[theorem]{Proposition}
\newtheorem{corollary}[theorem]{Corollary}

\newtheorem{definition}[theorem]{Definition}
\newtheorem{thm}{Theorem}

\theoremstyle{definition}

\allowdisplaybreaks
\begin{document}
\title[Zeros of polynomials of derivatives of zeta functions]{Zeros of polynomials of derivatives of zeta functions}
\author[T.~Nakamura]{Takashi Nakamura}
\address[T.~Nakamura]{Department of Liberal Arts, Faculty of Science and Technology, Tokyo University of Science, 2641 Yamazaki, Noda-shi, Chiba-ken, 278-8510, Japan}
\email{nakamuratakashi@rs.tus.ac.jp}
\urladdr{https://sites.google.com/site/takashinakamurazeta/}
\maketitle

\begin{abstract}
Let $P_s \in \mathcal{D}_s[X_0,X_1, \ldots,X_l]$ be a polynomial whose coefficients are the ring of all general Dirichlet series which converge absolutely in the half-plane $\Re (s) > 1/2$.
In the present paper, we show that the function $P_s(L(s), L^{(1)}(s),\ldots, L^{(l)}(s))$ has infinitely many zeros in the vertical strip $D:= \{ s \in {\mathbb{C}} : 1/2 < \Re (s) <1\}$ if $L(s)$ is hybridly universal and $P_s \in \mathcal{D}_s[X_0,X_1, \ldots,X_l]$ is a polynomial such that at least one of the degree of $X_1,\ldots,X_l$ is greater than zero. As a corollary, we prove that the function $(d^k / ds^k) P_s(L(s))$ with $k \in {\mathbb{N}}$  has infinitely many zeros in the strip $D$ when $L(s)$ is hybridly universal and $P_s \in \mathcal{D}_s[X]$ is a polynomial with degree greater than zero. The upper bounds for the numbers of zeros of $P_s(L(s), L^{(1)}(s),\ldots, L^{(l)}(s))$ and $(d^k / ds^k) P_s(L(s))$ are studied as well. 
\end{abstract}

\section{Introduction and the main results}

\subsection{Main results}
Let $D(s)$ be a general Dirichlet series of the form 
\begin{equation}\label{eq:gds}
D(s) := \sum_{n=1}^\infty a_n e^{-\lambda_n s}, \qquad a_n \in {\mathbb{C}}, \quad \lambda_n \in {\mathbb{R}} .
\end{equation}
Let us denote by $\mathcal{D}_s$ the ring of all general Dirichlet series defined as (\ref{eq:gds}) which are absolutely convergent in the half-plane $\Re (s) > 1/2$. Let ${\mathbb{N}}$ be the set of integers greater than $0$, and $P_s \in \mathcal{D}_s[X]$ be a polynomial whose coefficients are in $\mathcal{D}_s$. Then we have the following results which are proved in Section 2. Note that the definition of the hybrid universality is written in the next subsection. Some remarks and examples related to main results are given in Section 3.

\begin{theorem}\label{th:1}
Let $L(s)$ be hybridly universal and $P_s \in \mathcal{D}_s [X_0,X_1, \ldots,X_l]$ be a polynomial such that at least one of the degree of $X_1,\ldots,X_l$ is greater than zero. Then, for any $1/2 < \sigma_1 < \sigma_2 < 1$, there exists a constant $C_1>0$ such that for sufficiently large $T$, the function $P_s(L(s), L^{(1)}(s),\ldots, L^{(l)}(s))$ has more than $C_1T$ nontrivial zeros in the rectangle $\sigma_1 < \sigma < \sigma_2$, $0 < t <T$. 
\end{theorem}

\begin{corollary}\label{cor:1}
Suppose that a function $L(s)$ is hybridly universal, and $P_s \in \mathcal{D}_s[X]$ is a polynomial with degree greater than zero. Then, for any $1/2 < \sigma_1 < \sigma_2 < 1$, there exists a constant $C_1>0$ such that for sufficiently large $T$, the function $(d^k / ds^k) P_s(L(s))$ with $k \in {\mathbb{N}}$ has more than $C_1T$ nontrivial zeros in the rectangle $\sigma_1 < \sigma < \sigma_2$, $0 < t <T$. 
\end{corollary}

\begin{proposition}\label{pro:1}
For each $0 \le j \le l$, let $L^{(j)}(s)$ be a function defined as a Dirichlet series for $\sigma>1$ which can be continued analytically to a meromorphic function on $\Re (s)>1/2$ with a finite number of poles and all of them lie on the line $\Re (s)=1$. Moreover for each $0 \le j \le l$, suppose that $L^{(j)}(s)$ is a function of finite order and for any fixed $1/2< \Re (s)<1$, the square mean-value satisfies
\begin{equation}\label{eq:1pro1}
\int_0^T \bigl|L^{(j)}(\sigma+it)\bigr|^2dt \ll T, \quad \mbox{as} \quad T\to\infty .
\end{equation}
Then, for any $\sigma_0>1/2$ and any polynomial $P_s \in {\mathcal{D}}_s [X_0,X_1, \ldots,X_l]$, there exists a constant $C_2>0$ such that for sufficiently large $T$, the number of zeros $\rho$ with multiplicities of $P_s(L(s), L^{(1)}(s),\ldots, L^{(l)}(s))$ with $\Re (\rho)>\sigma_0$ and $0<\Im(\rho)<T$ is less than $C_2T$.
\end{proposition}

\begin{corollary}\label{cor:2}
Let $L^{(j)}(s)$ satisfy all assumptions in Proposition \ref{pro:1}. Then, for any $\sigma_0>1/2$ and any polynomial $P_s \in {\mathcal{D}}_s [X_0]$, there exists a constant $C_2>0$ such that for sufficiently large $T$,  the number of zeros $\rho$ with multiplicities of $(d^k / ds^k) P_s(L(s))$ with $\Re (\rho)>\sigma_0$ and $0<\Im(\rho)<T$ is less than $C_2T$.
\end{corollary}

It should be mentioned that Theorem \ref{th:1} and Proposition \ref{pro:1} are generalizations of Main Theorems 1 and 2 in \cite{napa4}, respectively (see Theorem \ref{th:mthm1}). 

\subsection{Hybrid universality}
In 1975, Voronin \cite{Voronin} showed the following universality theorem for the Riemann zeta-function $\zeta (s) := \sum_{n=1}^\infty n^{-s}$. In order to state it, we need some notation. Let $D := \{ s \in {\mathbb{C}} : 1/2 < \Re (s) <1\}$. Denote by $\mu(A)$ the Lebesgue measure of the set $A$, and, for $T>0$, write $\nu_T \{ \cdots \} := T^{-1} \mu \{ \tau \in [0,T] : \cdots \}$ where the dots stand for a condition satisfied by $\tau$. Then the modern version of the Voronin theorem can be written as follows.
\begin{thm}\label{thm:Voronin}
Let $K\subset D$ be a compact set with connected complement, and $f(s)$ be a non-vanishing continuous function on $K$ which is analytic in the interior of $K$.
Then for any $\varepsilon>0$, it holds that
$$
\liminf_{T \to \infty} \nu_T \Bigl\{ \max_{s \in K} \bigl| \zeta(s+i\tau)-f(s) \bigr| <\varepsilon \Bigr\}>0.
$$
\end{thm}
Roughly speaking, this theorem implies that any non-vanishing analytic function can be uniformly approximated by the Riemann zeta-function. In the current era, it is known that many zeta and $L$-functions, for example, Dirichlet $L$-functions, Dedekind zeta functions, Lerch zeta functions and $L$-functions associated with newforms, have the universality property (see for instance \cite{Lali}, \cite{MatsuS14} and \cite{St1877}). 

Combining the Kronecker-Weyl theorem on diophantine approximations and Voronin's universality theorem, we can define the hybrid universality. Denote the distance to the nearest integer by $\Vert\cdot\Vert$. The precise definition is as follows.
\begin{definition}
\emph{Hybrid universality} for the function $L(s)$ is the following property: Let $K$ and $f(s)$ be the same as in Theorem \ref{thm:Voronin}. Suppose that $\{\alpha_j\}_{1\le j \le k}$ are real numbers linearly independent over $\mathbb{Q}$. Then for any $\varepsilon>0$ and any real numbers $\{\theta_j\}_{1\le j \le k}$, 
$$
\liminf_{T\to\infty} \nu_T \Bigl\{ 
\max_{s \in K} \bigl| L(s+i\tau) - f(s) \bigr| < \varepsilon , \quad
\Vert\tau\alpha_j - \theta_j\Vert < \varepsilon, \,\,\, 1\le j \le k
\Bigr\}>0.
$$
\end{definition}
As mentioned in \cite{napa4}, the first result on hybrid universality was proved in weaker form by Gonek \cite{Gonek} and improved  by Kaczorowski and Kulas \cite{KaczorowskiKulas}. They showed that Dirichlet $L$-functions $L(s,\chi) := \sum_{n=1}^\infty \chi (n) n^{-s}$ satisfy the inequality in above definition for $\alpha_n = \log{p_n}$, where $p_n$ denotes the $n$-th prime number. Afterwards, Pa\'nkowski \cite{PankowskiHybrid} proved the hybrid universality in the most general form for an axiomatically defined wide class of $L$-functions having Euler product.  By using the hybrid universality, Nakamura and Pa\'nkowski showed the following Theorem in \cite[Main Theorem 1]{napa4}. 

\begin{thm}\label{th:mthm1}
Suppose that a function $L(s)$ is hybridly universal, $P_s \in \mathcal{D}_s[X]$ is not a monomial but a polynomial with degree greater than zero. Then, the function $P_s(L(s))$ has infinitely many zeros in $D$. More precisely, for any $1/2 < \sigma_1 < \sigma_2 < 1$, there exists a constant $C>0$ such that for sufficiently large $T$, the function $P_s(L(s))$  has more than $CT$ nontrivial zeros in the rectangle $\sigma_1 < \sigma < \sigma_2$, $0 < t <T$. 
\end{thm}

Moreover, they gave an upper bound for the number of zeros of certain polynomials of $L$-functions in \cite[Main Theorem 2]{napa4}. We can say that Theorem \ref{th:1} and Proposition \ref{pro:1} are multidimensional cases of Main Theorems 1 and 2 in \cite{napa4}, respectively. It should be mentioned that they also showed that a lot of zeta or $L$-functions have infinitely many complex zeros off the critical line in \cite[Section 3]{napa4} (see also Section 3.2). 

\subsection{Zeros of the derivatives of the Riemann zeta function}
Zeros of the derivatives of the Riemann zeta function have been investigated by many mathematicians. In 1935, Speiser \cite{Spei} showed that the Riemann Hypothesis is equivalent to $\zeta '(s)$ having no zeros in the vertical strip $0 < \Re (s) <1/2$. Let $N_k (T)$, $k\ge 1$, denote the number of zeros $\beta_k+i\gamma_k$ of the $k$-th derivative of the Riemann zeta function $\zeta^{(k)} (s)$ such that $0<\gamma_k<T$. Berndt \cite[p.~577]{Ber} proved that
$$
N_k (T) = \frac{T \log T}{2\pi} - \frac{1+\log 4\pi}{2\pi}T+O(\log T), \qquad T \to \infty.
$$ 
Levinson and Montgomery \cite{LM} obtained several results on the zeros of $\zeta^{(k)} (s)$. For example, they showed 
\begin{equation*}
\begin{split}
2\pi \sum_{0< \gamma_k \le T} (\beta_k-1/2) = & \, k T \log \log(T/2\pi) - 2\pi k {\rm{Li}} (T/2\pi) \\
&+ (\log 2 - 2k \log \log 2) (T/2) + O(\log T),
\end{split}
\end{equation*} 
where ${\rm{Li}} (x) := \int_2^x dy /\log y$ in \cite[Theorem 10]{LM}. They also proved in \cite[Theorem 7]{LM} that if $\zeta^{(j)}$, $j \in {\mathbb{N}} \cup \{0\}$ has only a finite number of non-real zeros in $\Re (s) <1/2$, then $\zeta^{(j+k)}$ has the same property for $k\ge 1$. 

On the other hand, Launrin\v cikas \cite{Lau85} showed the following. For any $\sigma_1$ and $\sigma_2$ such that $1/2 < \sigma_1 < \sigma_2 < 1$, there exists a $C = C(\sigma_1, \sigma_2) > 0$ such that for sufficiently large $T$ one can find more than $CT$ zeros of the function $\zeta'(s)$ lying in the domain $\sigma_1 < \Re (s) <\sigma_2$ and $0 < \Im (s) <T$. In \cite[Theorem 2]{Mey}, Meyrath proved that $\zeta^{(k)}(s)$, $k \ge 1$ has the strong universality, namely, $\zeta^{(k)}(s)$ can approximate any analytic function (with zeros) in the sense of the universality (see for example \cite[p.~30]{St1877} for the definition of the strong universality). Therefore, we have the following. 
\begin{thm}\label{thm:c}
For any $\sigma_1$ and $\sigma_2$ with $1/2 < \sigma_1 < \sigma_2 < 1$, there exists a constant $C$ such that for sufficiently large $T$ one can find more than $CT$ zeros of the function $\zeta^{(k)}(s)$ with $k \in {\mathbb{N}}$ lying in $\sigma_1 < \Re (s) <\sigma_2$ and $0 < \Im (s) <T$. 
\end{thm}

Furthermore, Launrin\v cikas \cite[p.~200]{Lau12} showed that $\sum_{k=1}^l a_k \zeta^{(k)}(s)$, where $a_k \in {\mathbb{C}} \setminus \{ 0\}$ has the strong universality (see also \cite[Theorem 3.1]{Lau12}). Hence $\sum_{k=1}^l a_k \zeta^{(k)}(s)$ with $a_k \in {\mathbb{C}} \setminus \{ 0\}$ has more than $CT$ zeros in $\sigma_1 < \Re (s) <\sigma_2$ and $0 < \Im (s) <T$. Obviously, Theorem \ref{th:1} is a generalization of this result and Theorem \ref{thm:c}. 

\section{Proofs}
\subsection{Preliminary}
We quote the first lemma from \cite[Lemma 6.6.1]{Lali} or \cite[p.~194]{St1877}. We can find the second lemma in \cite[Lemma 1]{NaPa1} and \cite[Lemma 2.1]{NaPaE}.
\begin{lemma}\label{lem:661}
Let $c_0 \ne 0, c_1, \ldots , c_m$ be complex numbers. Then there exists a polynomial $p_m (s) = b_0 + b_1s + \cdots +b_ms^m$ such that
$$
c_n = \frac{d^n}{ds^n} \Bigl[ \exp \bigl( p_m (s) \bigr) \Bigr]_{\!s=0\,} , \qquad 0 \le n \le m.
$$
\end{lemma}
\begin{lemma}\label{lem:AlmostPer}
Let $u,v \in {\mathbb{N}}$ and $D_u(s) = \sum_{n=1}^\infty a_{u,n} e^{-\lambda_{u,n} s}$ be a general Dirichlet series for $1\le u \le v$. Then, for every $\varepsilon>0$ and every compact set $K$ lying in the half-plane of absolute convergence, there exist $\delta > 0$, $M \in {\mathbb{N}}$ and a finite set $I\in\{1,2,\ldots,v\} \times {\mathbb{N}}$ such that the numbers $\lambda_{j,k}$, where $(j,k) \in I$, are linearly independent over $\mathbb{Q}$ and, moreover, if 
$$
\max_{(j,k)\in I} \Bigl\Vert \frac{\tau\lambda_{j,k}}{2 \pi M} \Bigr\Vert < \delta
$$
then it holds that
$$
\max_{1\le u \le v} \max_{s \in K} \bigl|D_u(s+i\tau) - D_u(s)\bigr| < \varepsilon.
$$
\end{lemma}
\begin{proof}
For the reader's convenience, we write the proof here. Take $\varepsilon > 0$ and fix a compact set $K$. Then there exists a positive integer $N$ such that
$$
\max_{1\le u \le v}\max_{s\in K} \sum_{n>N} |a_{u,n}| e^{-\lambda_{u,n} \Re (s)} < \varepsilon
$$
since the Dirichlet series $D_u(s)$ converges absolutely on the compact set $K$. 
Now take the set $I$ of indices such that $\{\lambda_{j,k} : (j,k)\in I\}$ is the basis of a vector space over $\mathbb{Q}$ generated by all numbers $\lambda_{u,n}$ with $n \le N$. Moreover, let $M$ be a positive integer such that all these numbers can be expressed as a linear combination of elements $\lambda_{j,k}/M$, where $(j,k)\in I$, with integer coefficients. Then, for $s\in K$ and $1\le u\le v$, one has
\begin{align*}
|P_u(s+i\tau)  - P_u(s)|& \le 
\sum_{n \le N} \left|a_{u,n}\right| \left|e^{-\lambda_{u,n} (s+i\tau)}-e^{-\lambda_{u,n} s} \right| 
+ 2\sum_{n>N} \left|a_{u,n}\right|  e^{-\lambda_{u,n} \Re (s)}\\
&\ll \max_{n\le N}\left|e^{-i\tau\lambda_{u,n}} - 1\right| + 2\varepsilon \ll 
\max_{(j,k)\in I} \biggl\Vert \frac{\tau\lambda_{j,k}}{2 \pi M} \biggr\Vert  + 2\varepsilon.
\end{align*}
By taking suitable $\delta > 0$, we complete the proof.
\end{proof}

\begin{lemma}\label{lem:1}
Let ${\mathcal{K}}$, ${\mathcal{K}}'$ and ${\mathcal{K}}''$ be a closed disks whose centers are $a \in {\mathbb{C}}$ and radiuses are $r$, $r'$, $r''$ which satisfy $0<r<r'<r''$. Suppose that $f(s)$ and $g(s)$ are analytic in the interior of ${\mathcal{K}}''$ and satisfy $\max_{s \in {\mathcal{K}}'} |f(s)-g(s)|< \varepsilon$. Then it holds that
\begin{equation}\label{eq:decau1}
\max_{s \in {\mathcal{K}}} \bigl|f^{(k)}(s)-g^{(k)}(s)\bigl| < \frac{k! 2^k \varepsilon}{(r'-r)^k} .
\end{equation}
\end{lemma}
\begin{proof}
Put $r_\# := (r'-r)/2$. Note that $s+r_\#e^{i\theta} \in {\mathcal{K}}'$ for all $0\le \theta < 2\pi$ when $s \in {\mathcal{K}}$. From Cauchy's integral formula, we have
\begin{equation*}
\begin{split}
&\bigl|f^{(k)}(s)-g^{(k)}(s)\bigl| = 
\biggl|\frac{k!}{2\pi i} \int_{|z-s|=r_\#} \frac{f(z)-g(z)}{(z-s)^{k+1}} dz \biggr| \\ = & \,
\left|\frac{k!}{2\pi i} \int_0^{2\pi} 
\frac{f(s+r_\#e^{i\theta})-g(s+r_\#e^{i\theta})}{(r_\#e^{i\theta})^{k+1}} r_\# e^{i\theta} d\theta \right|
< \frac{k! \varepsilon}{2\pi} \int_0^{2\pi} \frac{d\theta}{r_\#^k},
\end{split}
\end{equation*}
for any $s \in {\mathcal{K}}$. This inequality implies (\ref{eq:decau1}). 
\end{proof}

\subsection{Proof of main results }
\begin{proof}[Proof of Theorem \ref{th:1}]
For any $\varepsilon >0$, any closed disk $K' \subset D$ and any function $f(s)$ which is non-vanishing and continuous on $K''$ and analytic in the interior of $K''$, where $K'' \supset K'$, there exists $\tau \in {\mathbb{R}}$ such that
$$
\max_{1\le u \le v} \max_{s \in K'} \bigl|D_u(s+i\tau) - D_u(s)\bigr| < \varepsilon \quad \mbox{and} \quad
\max_{s \in K'} \bigl|L(s+i\tau) - f(s)\bigr| < \varepsilon
$$
from Lemma \ref{lem:AlmostPer} and the definition of the hybrid universality. According to Lemma \ref{lem:1}, for $K \subset K'$, there exists $\tau \in {\mathbb{R}}$ such that
\begin{equation}\label{eq:1}
\max_{1\le u \le v} \max_{s \in K} \bigl|D_u(s+i\tau) - D_u(s)\bigr| < \varepsilon \quad \mbox{and} \quad
\max_{0 \le j \le l} \max_{s \in K} \bigl|L^{(j)}(s+i\tau) - f^{(j)}(s)\bigr| < \varepsilon.
\end{equation}

First suppose 
\begin{equation}\label{eq:mono}
P_s \bigl(L(s), L^{(1)}(s),\ldots, L^{(l)}(s)\bigr) = D(s) L^{(k)}(s), 
\end{equation}
where $1 \le k \le l$ and $D(s)$ is a general Dirichlet series converges absolutely when $\Re (s) >1/2$. Now let $1/2 < \sigma_1 < \sigma_2<1$, $K \subset \{ s \in {\mathbb{C}} : \sigma_1< \Re (s) < \sigma_2 \}$ be a closed disk whose center is $\alpha \in {\mathbb{C}}$ and $f(s) := se^{\eta s}$ with $\eta \ne 0$. Obviously, the function $se^{\eta s}$ is non-vanishing and continuous on $K$ and analytic in the interior of $K$. In addition, one has
$$
f^{(k)}(s) = \frac{d^k}{ds^k} se^{\eta s} = k \eta^{k-1} e^{\eta s} + \eta^k s e^{\eta s} =
\eta^{k-1} e^{\eta s} (k+ \eta s)
$$
by the principle of mathematical induction. Put $\eta = -k /\alpha$. Then the function 
\begin{equation}\label{eq:Am}
A(s,\alpha , k) := \frac{d^k}{ds^k} se^{-k s/\alpha} = 
(-1)^{k-1} k^{k-1} \alpha^{-k} e^{-k s/\alpha} (k\alpha-ks)
\end{equation}
has a only one zero at $s=\alpha$ when $s \in K$. From (\ref{eq:1}), we have
$$
\max_{|s-\alpha|=r} \bigl| D(s+i\tau) L^{(k)}(s+i\tau) - D(s) f^{(k)}(s) \bigr| < 
\min_{|s-\alpha|=r} \bigl|  D(s) f^{(k)}(s) \bigr|
$$
if $D(s)$ does not vanish on the circle $|s-\alpha|=r$. By applying Rouch\'e's theorem, we can show that whenever the inequality above holds, the function $D(s+i\tau) L^{(k)}(s+i\tau)$ has a zero in the interior of $K$. Moreover, the measure of such $\tau \in [0,T]$ is greater than $cT$ for some constant $c>0$ according to the hybrid universality. Hence, for some positive constant $C_1$, the function $D(s) L^{(k)}(s)$ has more than $C_1T$ nontrivial zeros in the rectangle $\sigma_1 < \sigma < \sigma_2$, $0 < t <T$ when $1 \le k \le l$. 

Assume that $P_s \in \mathcal{D}_s [X_0,X_1, \ldots,X_l]$ is a monomial such that at least one of the degree of $X_1,\ldots,X_l$ is greater than zero. Then a factor of $P_s(L(s), L^{(1)}(s),\ldots, L^{(l)}(s))$ is expressed as the form (\ref{eq:mono}). Therefore, the function $P_s(L(s), L^{(1)}(s),\ldots, L^{(l)}(s))$, where $P_s \in \mathcal{D}_s [X_0,X_1, \ldots,X_l]$ is a monomial such that at least one of the degree of $X_1,\ldots,X_l$ is greater than zero, has more than $C_1T$ nontrivial zeros in $\sigma_1 < \sigma < \sigma_2$, $0 < t <T$. \\

Next suppose $P_s \in \mathcal{D}_s [X_0,X_1, \ldots,X_l]$ is a non-monomial polynomial. Then we can see that at least one of the degree of $X_0,\ldots,X_l$ is greater than zero. It should be noted that we show the case that not at least one of the degree of $X_1,\ldots,X_l$ is greater than zero but at least one of the degree of $X_0,X_1,\ldots,X_l$ is greater than zero since the latter case needs less assumptions and contains \cite[Main Theorem 1]{napa4} or Theorem \ref{th:mthm1}. In this case, the function $P_s(L(s), L^{(1)}(s),\ldots, L^{(l)}(s))$ can be written by
\begin{equation}\label{eq:2}
\begin{split}
&P_s \bigl(L(s), L^{(1)}(s),\ldots, L^{(l)}(s)\bigr) = \\ &
\sum_{0 \le d_j \le n_j, \, 0 \le j \le l} D_{d_0d_1\ldots d_l} (s) 
\bigl( L^{(0)}(s) \bigr)^{d_0} \bigl( L^{(1)}(s) \bigr)^{d_1} \cdots \bigl( L^{(l)}(s) \bigr)^{d_l}
\end{split}
\end{equation}
where $D_{d_0d_1\ldots d_l} (s)$ are general Dirichlet series converge absolutely when $\Re (s) >1/2$ and at least two of $D_{d_0d_1\ldots d_l} (s)$ with $0 \le d_j \le n_j$ for $0 \le j \le l$ are not identically vanishing. Let $K \subset \{ s \in {\mathbb{C}} : 1/2 < \sigma_1< \Re (s) < \sigma_2<1 \}$ be a closed disk whose center is $\alpha \in {\mathbb{C}}$ and let at least two of $D_{d_0d_1\ldots d_l} (\alpha)$ with $0 \le d_j \le n_j$ for $0 \le j \le l$ be not zero. This is possible from the assumption that $P_s \in \mathcal{D}_s [X_0,X_1, \ldots,X_l]$ is a non-monomial polynomial. Then we can choose non-zero complex numbers $\theta_0, \theta_1, \ldots, \theta_l$ which satisfy 
\begin{equation}\label{eq:zerotheta}
0= \sum_{0 \le d_j \le n_j, \, 0 \le j \le l} D_{d_0d_1\ldots d_l} (\alpha) \,
\theta_0^{d_0} \theta_1^{d_1} \cdots \theta_l^{d_l} .
\end{equation}
For $0 \le j \le l$, we can take $\eta_0,\eta_1, \ldots, \eta_l \in {\mathbb{C}}$ satisfying
$$
\theta_j = \frac{d^j}{ds^j} \Bigl[ f(s) \Bigr]_{\!s=\alpha\,} ,
\qquad f(s) := \exp \bigl( \eta_0 + \eta_1 (s - \alpha) + \cdots + \eta_l (s-\alpha)^l \bigr)
$$
from Lemma \ref{lem:661}. Hence $A(s, \alpha \,; f)$ written by
\begin{equation}\label{eq:Ap}
\begin{split}
&A(s, \alpha \,; f) := \sum_{0 \le d_j \le n_j, \, 0 \le j \le l} D_{d_0d_1\ldots d_l} (s) 
\bigl( f^{(0)}(s) \bigr)^{d_0} \bigl( f^{(1)}(s) \bigr)^{d_1} \cdots \bigl( f^{(l)}(s) \bigr)^{d_l} 
\end{split}
\end{equation} 
has a zero at the point $s=\alpha$ from (\ref{eq:zerotheta}) and the definition of $f(s)$. For simplicity, put
\begin{equation}\label{eq:Z}
Z(s+i\tau) := P_{s+i\tau} \bigl(L(s+i\tau), L^{(1)}(s+i\tau),\ldots, L^{(l)}(s+i\tau)\bigr) .
\end{equation}
By using (\ref{eq:1}), we have
\begin{equation*}
\begin{split}
&\max_{|s-\alpha|=r} \Biggl| Z(s+i\tau) - 
\sum_{0 \le d_j \le n_j, \, 0 \le j \le l} D_{d_0d_1\ldots d_l} (s) 
\bigl( f^{(0)}(s) \bigr)^{d_0} \bigl( f^{(1)}(s) \bigr)^{n_1} \cdots \bigl( f^{(l)}(s) \bigr)^{d_l} \Biggr| \\ & < 
\min_{|s-\alpha|=r} \Biggl| \sum_{0 \le d_j \le n_j, \, 0 \le j \le l} D_{d_0d_1\ldots d_l} (s) 
\bigl( f^{(0)}(s) \bigr)^{d_0} \bigl( f^{(1)}(s) \bigr)^{d_1} \cdots \bigl( f^{(l)}(s) \bigr)^{d_l} \Biggr|
\end{split}
\end{equation*}
when the function $A(s, \alpha \,; f)$ does not vanish on the circle $|s-\alpha|=r$. Therefore, the function $Z(s+i\tau)$ defined by (\ref{eq:Z}) also has a zero in the interior of $K$ by Rouch\'e's theorem. Thus $P_s(L(s), L^{(1)}(s),\ldots, L^{(l)}(s))$ has more than $C_1T$ nontrivial zeros in the rectangle $\sigma_1 < \sigma < \sigma_2$, $0 < t <T$ by the hybrid universality when $P_s \in \mathcal{D}_s [X_0, X_1, \ldots, X_l]$ is a non-monomial polynomial.
\end{proof}

\begin{proof}[Proof of Proposition \ref{pro:1}]
The function $P_s(L(s), L^{(1)}(s),\ldots, L^{(l)}(s))$ can be continued analytically to a meromorphic function on $\Re (s)>1/2$ with a finite number of poles and all of them lie on the line $\Re (s)=1$ according to (\ref{eq:2}) and the assumptions of $L^{(j)}(s)$. Moreover, $P_s(L(s), L^{(1)}(s),\ldots, L^{(l)}(s))$ is a function of finite order from (\ref{eq:2}) and the assumption that $L^{(j)}(s)$ are function of finite order for all $0 \le j \le l$. 

Put $n := (l+1)^2 n_0n_1\cdots n_l$. Due to H\"older's inequality, one has
\begin{equation*}
\begin{split}
&\int_0^T \Bigl| D_{d_0d_1\ldots d_l}(\sigma+it) 
\bigl( L^{(0)}(\sigma+it) \bigr)^{d_0}  \cdots \bigl( L^{(l)}(\sigma+it) \bigr)^{d_l} \Bigr|^{2/n} dt \\ &
\ll \int_0^T \bigl|  L^{(0)}(\sigma+it) \cdots L^{(l)}(\sigma+it) \bigr|^{2/(l+1)} dt \\
&\le \biggl( \int_0^T \bigl|  L^{(0)}(\sigma+it) \bigr|^{2} dt \biggr)^{1/(l+1)} \cdots
\biggl( \int_0^T \bigl|  L^{(l)}(\sigma+it) \bigr|^{2} dt \biggr)^{1/(l+1)}  .
\end{split}
\end{equation*}
Hence, by the assumption (\ref{eq:1pro1}), it holds that
$$
\sum_{0 \le d_j \le n_j, \, 0 \le j \le l} \int_0^T \Bigl| D_{d_0d_1\ldots d_l}(\sigma+it) 
\bigl( L^{(0)}(\sigma+it) \bigr)^{d_0} \cdots \bigl( L^{(l)}(\sigma+it) \bigr)^{d_l} \Bigr|^{2/n} dt \ll T.
$$
Therefore, we obtain Proposition \ref{pro:1} by modifying the proof of \cite[Main Theorem 2]{napa4} (see also the proofs of \cite[Theorems 7.1 and 7.6]{St1877}).
\end{proof}

\begin{proof}[Proof of Corollaries \ref{cor:1} and \ref{cor:2}]
Let $D(s)$ be a general Dirichlet series converges absolutely in the half-plane $\Re (s) >1/2$. Then it is well-known that the general Dirichlet series of $D^{(k)}(s)$ also converges absolutely when $\Re (s) >1/2$ for any $1 \le k \le l$ (see for instance \cite[Theorem 4]{HaRi}). Hence $(d^k / ds^k) P_s(L(s))$ can be expressed as (\ref{eq:2}). Thus Corollaries \ref{cor:1} and \ref{cor:2} are deduced by Theorem \ref{th:1} and Proposition \ref{pro:1}, respectively.
\end{proof}

\section{Remarks and Examples}
\subsection{Remarks}
We have the following remarks.

\subsubsection*{(1)} In general, polynomials $P_s (L(s))$ are not universal even if $L(s)$ is hybridly universal. As an example, consider the non-universality of the following function
$$
F(s) = \bigl( 1-9^{9(s-3/4)} \bigr) \zeta (s) +2
$$
in the close disk ${\mathcal{K}}$ whose center is $3/4$ and the radius is $\pi / \log 9^9 < r <1/4$. Note that $\pi / \log (9^9) = 0.158867... < 1/4$. Then every shifted function $F(s+i\tau)-2$ has at least one zero inside ${\mathcal{K}}$, and hence it can not approximate uniformly functions which do not vanish inside ${\mathcal{K}}$. Especially, the function $F(s+i\tau)-2$ can not approximate the constant function $1$ (we can find this argument in \cite[p.~114]{Kazero}). 
Therefore, the function $F(s)$ can not approximate the constant function $3$ in the sense of universality. Despite of the fact above, we can show that $F(s)$ has zeros in the strip $1/2 < \sigma <1$ by Theorem \ref{th:1}.

\subsubsection*{(2)} In \cite[Theorem 4]{Gonek2}, Gonek, Lester and Milinovich proved that a positive proportion of the $a$-points of $\zeta (s)$ are simple in fixed strips to the right of the line $\Re(s)= 1/2$. As a generalization of this result, we consider simple zeros of $P_s(L(s), L^{(1)}(s),\ldots, L^{(l)}(s))$. Suppose that $A(s,\alpha, k)$ or $A(s,\alpha \,; f)$, defined by (\ref{eq:Am}) or (\ref{eq:Ap}) respectively, has only a simple zero in the disk $K$ whose center is $\alpha$ and the radius is $r$. Then $Z(s+i\tau)$ written by (\ref{eq:Z}) also has a simple zero in the disk $K$ from Rouch\'e's theorem. Hence there is a positive constant $c_1$ such that at least $c_1T$ of zeros of $P_s(L(s), L^{(1)}(s),\ldots, L^{(l)}(s))$ in the region $(\sigma_1, \sigma_2) \times (0, T)$ are simple in this case. It should be mentioned that $A(s,\alpha \,; f)$ does not have simple zeros when $P_s(L(s), L^{(1)}(s),\ldots, L^{(l)}(s))$ can be expressed as $(\zeta(s) -1)^2$, $(\zeta^2 (s)-\zeta (2s))^3$, $(\zeta'(s) + \zeta''(s) + \zeta(s+1))^5$ and so on. 

\subsubsection*{(3)} The Riemann zeta function $\zeta (s)$ satisfies all assumptions in Proposition \ref{pro:1}. This is proved as follows. By the following partial summation
$$
\zeta (s) = \sum_{n\le N} \frac{1}{n^s} + \frac{N^{1-s}}{s-1} +s \int_N^{\infty} \frac{[u]-u}{u^{s+1}} du,
$$
where $[u]$ denotes the maximal integer less than or equal to $u$ (see for example \cite[(1.3)]{St1877}), we can see that $\zeta^{(j)}(s)$ is a function of finite order and for any $0 \le j \le l$ and any fixed $1/2< \Re (s)<1$. Moreover, it is known that if $\eta$ and $\theta$ are fixed and $\eta,\theta >-1/2$, $\eta+\theta>1$, then it holds that
$$
\int_1^\infty \zeta^{(u)} (\eta+it) \zeta^{(v)} (\theta-it) dt \sim \zeta^{(u+v)} (\eta+\theta)T,
$$
where $u$ and $v$ are non-negative integers (see for example \cite[(2)]{Ing}). Therefore, the derivatives of the Riemann zeta function $\zeta^{(j)}(s)$ fulfill (\ref{eq:1pro1}) for all $0 \le j \le l$. 

\subsubsection*{(4)} 
There exist some functions composed by $\zeta (s)$ do not vanish in the strip $D$. For example, the function
$$
\exp \bigl( \zeta (s) \bigr) =\sum_{n=0}^\infty \frac{\zeta^n(s)}{n!} = 
1+\zeta (s) + \frac{\zeta^2(s)}{2!} + \frac{\zeta^3(s)}{3!} + \cdots
$$
does not have zeros in the vertical strip $D$. Hence we can not remove the assumption that $P_s$ is a polynomial (with finite degree) in Theorem \ref{th:1}. Furthermore, we introduce functions $G(s), F_\pm (s) \in {\mathcal{D}}_s [\zeta(s),s]$ do not vanish when $s \in D$ in Examples (4) and (5). 

\subsection{Examples}
We have the following examples which we can apply the main results.

\subsubsection*{(1)} According to Corollaries \ref{cor:1} and \ref{cor:2}, the derivatives of zeta functions appeared in \cite[Section 3]{napa4}, for example, the zeta functions associated to symmetric matrices treated by Ibukiyama and Saito, certain spectral zeta functions and the Euler-Zagier multiple zeta functions, have more than $C_1T$ and less than $C_2T$ nontrivial zeros in some rectangle since these zeta functions can be expressed by the form (\ref{eq:2}). 

\subsubsection*{(2)} It is known that the second derivative of the logarithm of the Riemann zeta function $(\log \zeta (s))''=(\zeta' (s) / \zeta(s))'= (\zeta (s))^{-2}(\zeta (s)\zeta''(s)-\zeta'(s)^2)$ appears in the pair correlation for the Riemann zeta function (see for example Bogomolny and Keating \cite{BK} and Rodgers \cite{Rod}). Recently, Stopple \cite[Theorem 3]{Sto} showed that for positive $\delta$ the number of zeros of 
$$\mu (s):=\zeta (s)\zeta''(s)-\zeta'(s)^2$$ 
in the region $|\Im (s)|\le T$, $\Re (s) >5/6+\delta$ is less than $C(\delta)T$, where $C(\delta) >0$. By using Theorem \ref{th:1}, Proposition \ref{pro:1} and the fact mentioned in Remark (3), we can prove that the number of zeros of $\mu (s)$ in the region $|\Im (s)|\le T$, $\Re (s) >5/6$ is more than $C_1T$ and less than $C_2T$, where $C_2\ge C_1>0$. 

\subsubsection*{(3)} 
Let $m \in {\mathbb{N}}$ and $H(P)$, where $P \in {\mathbb{P}}^m ({\mathbb{Q}})$, be the height of the projective space ${\mathbb{P}}^m ({\mathbb{Q}})$. Then the height zeta function ${\mathcal{Z}}_m (s)$ is expressed as
$$
{\mathcal{Z}}_m (s) := \sum_{P \in {\mathbb{P}}^m ({\mathbb{Q}})} H(P)^{-s} = \frac{1}{\zeta (s)}
\sum_{n=0}^{[m/2]} \binom{m+1}{2n+1} 2^{m-2n} \zeta (s-m+2n)
$$
(see for example \cite[Section 3.4.3]{TB}). It is well-known that for $\Re (s) >1$, one has $(\zeta(s))^{-1} = \sum_{n=1}^\infty \mu (n) n^{-s}$. Note that the M\"obius function $\mu (n)$ is defined by $\mu(n)=0$ if $n$ is divisible by $p^2$ for some prime number $p$, and otherwise $\mu (n) = (-1)^{\omega (n)}$, where $\omega (n)$ is the number of distinct prime factors of $n$ (see for instance \cite[p.~156]{Cohen}). Therefore, the derivatives of zeta function ${\mathcal{Z}}_m (s)$ have more than $C_1T$ and less than $C_2T$ nontrivial zeros in the rectangle $\sigma_1 < \sigma < \sigma_2$, $0 < t <T$, where $m-1/2 <\sigma_1 < \sigma_2 <m$ when $m \ge 2$ from Corollaries \ref{cor:1} and \ref{cor:2}.

We can find some other zeta functions which can be written as (non-monomial) polynomials in  \cite[D-16, D-19, D-64, D-70 and D-71]{GS}.

\subsubsection*{(4)} In \cite[Theorem 1.3]{Napr}, it is proved that the function 
$$
G(s) := \zeta(s) + Cs \in {\mathcal{D}}_s [\zeta(s),s]
$$
does not vanish in the half-plane $\sigma >1/18$ when $|C| > 10$ and $-19/2 \le \Re (C) \le 17/2$. By Corollaries \ref{cor:1} and \ref{cor:2}, the derivatives of the function $(d^k/ds^k) (\zeta(s) + Cs)$ with $k \in {\mathbb{N}}$ has more than $C_1T$ and less than $C_2T$ complex zeros in the rectangle $1/2<\sigma_1 < \sigma < \sigma_2<1$, $0 < t <T$ for some positive constants $C_1  \le C_2$. 

\subsubsection*{(5)} All complex zeros of the function
$$
F_{\pm} (s) := s-1 \pm 2\pi \frac{\zeta (s-1)}{\zeta (s+1)} = 
\frac{\Gamma (s)}{\Gamma (s-1)} \pm 2\pi \frac{\zeta (s-1)}{\zeta (s+1)} \in {\mathcal{D}}_s [\zeta(s),s]
$$
are on the critical line $\Re (s) =1/2$ (see \cite[p.~262]{Cohen}). Hence the functions $F_{\pm} (s)$ satisfy an analogue of the Riemann hypothesis. Note that $1/\Gamma (s)$ is analytic on the whole complex plane and recall that $(\zeta(s))^{-1} = \sum_{n=1}^\infty \mu (n) n^{-s}$, where $\mu (n)$ is the M\"obius function. Therefore, the functions $(d^k/ds^k) F_{\pm} (s)$ with $k \in {\mathbb{N}}$ has more than $C_1T$ and less than $C_2T$ complex zeros in the rectangle $3/2<\sigma_3 < \sigma < \sigma_4<2$, $0 < t <T$.

\subsection*{Acknowledgments}
The author was partially supported by JSPS grant 24740029. 

 

\begin{thebibliography}{1}

\bibitem{Ber} B.~C.~Berndt, {\it{The number of zeros for $\zeta^{(k)} (s)$}}, Journal of the London Mathematical Society, {\bf{2}} (1970), 577--580. 

\bibitem{BK} E.~B.~Bogomolny and J.~P.~Keating, {\it{Gutzwiller's trace formula and spectral statistics: beyond the diagonal approximation}}, Phys.~Rev.~Lett. {\bf{77}} (1996), no.~8, 1472--1475.

\bibitem{Cohen} {\rm H.~Cohen}, {\em Number theory. Vol.~II. Analytic and modern tools} (Graduate Texts in Mathematics, 240. Springer, New York, 2007). 

\bibitem{Gonek} S.~M.~Gonek, \textit{Analytic Properties of Zeta and L-functions}, Ph.D. Thesis, Universality of Michigan (1979).

\bibitem{Gonek2} S.~M.~Gonek, S.~L.~Lester and M.~B.~Milinovich, {\it{A note on simple a-points of L-functions}},  Proc.~Amer.~Math.~Soc. {\bf{140}} (2012), no.~12, 4097--4103.

\bibitem{GS} 
H.~W.~Gould and T.~Shonhiwa, {\it{A Catalog of Interesting Dirichlet Series}}, Missouri J.~Math.~Sci. {\bf{20}}, Issue 1 (2008), 2--18.

\bibitem{HaRi} G.~H.~Hardy and M.~Riesz, {\it{The general theory of Dirichlet's series}}. Cambridge Tracts in Mathematics and Mathematical Physics, No.~18, Cambridge University press 1915. 

\bibitem{Ing} A.~E.~Ingham, {\it{Mean-value theorems in the theory of the Riemann zeta-function}}, Proc.~London Math.~Soc. Ser.~2, {\bf{27}} (1928), 273--300.

\bibitem{KaczorowskiKulas} J.~Kaczorowski and M.~Kulas, {\it{On the non-trivial zeros off line for $L$-functions from extended Selberg class}}, Monatshefte Math. \textbf{150} (2007), no.~3, 217--232.

\bibitem{Kazero} J.~Kaczorowski, {\it{Some remarks on the universality of periodic L-functions}}. New directions in value-distribution theory of zeta and $L$-functions, 113--120, Ber.~Math., Shaker Verlag, Aachen, 2009.

\bibitem{Lau85} A.~Laurin\v cikas, {\it{Zeros of the derivative of the Riemann zeta-function}}, Lithuanian Math.~J. {\bf{25}} (1985), no.~3, 255--260. 

\bibitem{Lali} 
A.~Laurin\v cikas, {\it{Limit Theorems for the Riemann Zeta-function}}, Mathematics and its Applications, {\bf{352}}. Kluwer Academic Publishers Group, Dordrecht, 1996.

\bibitem{Lau12} A.~Laurin\v cikas, {\it{Universality of composite functions}},  Functions in number theory and their probabilistic aspects, 191--204, {\it{RIMS Kokyuroku Bessatsu}}, {\bf{B34}}, Res.~Inst.~Math.~Sci. (RIMS), Kyoto, 2012.

\bibitem{LM}
N.~Levinson and H.~L.~Montgomery, {\it{Zeros of the derivatives of the Riemann zeta function}}, Acta Math. {\bf{133}} (1974), 49--65. 

\bibitem{Mey}
T.~Meyrath, {\it{On the universality of derived functions of the Riemann zeta-function}}, J.~Approx.~Theory {\bf{163}} (2011),  no.~10, 1419--1426. 

\bibitem{MatsuS14}
K.~Matsumoto, {\it{A survey on the theory of universality for zeta and $L$-functions}}, Number theory, 95--144, Ser.~Number Theory Appl., {\bf{11}}, World Sci.~Publ., Hackensack, NJ, 2015. (arXiv:1407.4216).

\bibitem{NaPa1}
T.~Nakamura and \L. Pa\'{n}kowski, {\it{On universality for linear combinations of $L$-functions}}, Monatshefte fuer Mathematik. {\bf{165}} (2012), no.~3, 433--446.

\bibitem{napa4}
T.~Nakamura and \L. Pa\'{n}kowski, {\it{On complex zeros off the critical line for non-monomial polynomial of zeta-functions}}, to appear in Mathematische Zeitschrift (arXiv:1212.5890).

\bibitem{NaPaE}
T.~Nakamura and \L. Pa\'{n}kowski, {\it{On zeros and \textit{c}-values of Epstein zeta-functions}}, {\it{\v{S}iauliai Mathematical Seminar}} (Special issue celebrating the 65th birthday of Professor Antanas Laurin\v{c}ikas) {\bf{8}} (2013), 181--195.

\bibitem{Napr} T.~Nakamura, {\it{A modified Riemann zeta distribution in the critical strip}}, Proceedings of the American Mathematical Society, {\bf{143}} (2015), no.~2, 897--905. 

\bibitem{PankowskiHybrid} \L.~Pa\'{n}kowski, {\it{Hybrid joint universality theorem for the Dirichlet L-functions}}, Acta Arith. {\bf 141} (2010), no.~1, 59--72.

\bibitem{Rod} B.~Rodgers, {\it{Macroscopic pair correlation of the Riemann zeroes for smooth test functions}}, The Quarterly Journal of Mathematics, (2012), 1--23. doi:10.1093/qmath/has024.

\bibitem{Spei} A.~Speiser, {\it{Geometrisches zur Riemannschen Zetafunktion}}, Mathematische Annalen, {\bf{110}} (1935),  no.~1, 514--521. 

\bibitem{St1877} J.~Steuding, \textit{Value-Distribution of L-functions}, Lecture Notes in Mathematics, 1877, Springer, Berlin (2007).

\bibitem{Sto} 
J.~Stopple, {\it{Notes on $\log (\zeta(s))''$}}, arXiv:1311.5465 (to appear in the Rocky Mountain Journal of Mathematics).

\bibitem{TB} 
R.~Takloo-Bighash, {\it{Points of bounded height on algebraic varieties}}, (Unpublished course notes), \url{http://homepages.math.uic.edu/~rtakloo/papers/ipm/ipm.pdf}

\bibitem{Voronin} S.~M.~Voronin, {\it{Theorem on the universality of the Riemann zeta function}}, Izv Akad Nauk SSSR Ser Mat. \textbf{39} (1975) 475--486 (in Russian); Math USSR Izv \textbf{9} (1975) 443--453.

\end{thebibliography}
\end{document}